\documentclass[a4paper,11pt]{article}

\usepackage[a4paper]{geometry}
\usepackage{color,graphicx,psfrag}
\usepackage{bbm}
\usepackage{amssymb,amsmath,amsthm}
\usepackage{authblk}
\usepackage[all,cmtip]{xy}

\newcommand{\R}{\mathbbm{R}}

\newcommand{\F}{\mathbbm{F}}

\DeclareMathOperator{\supp}{supp}
\newcommand{\I}{\mathbbm{1}}

\theoremstyle{plain}
\newtheorem{theorem}{Theorem}

\newtheorem{lemma}[theorem]{Lemma}

\theoremstyle{definition}
\newtheorem{definition}[theorem]{Definition}
\newtheorem{remark}[theorem]{Remark}

\newtheorem{remarks}[theorem]{Remarks}
\newtheorem{example}[theorem]{Example}

\theoremstyle{remark}

\newcommand{\map}{f}
\newcommand{\faces}[2]{\Sigma_{#1}(#2)}

\newcommand{\incident}[2]{[#1\colon #2]}
\newcommand{\expfactor}{\eta}
\newcommand{\sysbound}{\vartheta}
\newcommand{\overlap}{\mu}
\newcommand{\isobound}{L}
\newcommand{\sparsity}{\varepsilon}
\DeclareMathOperator{\im}{im}

\title{On Expansion and Topological Overlap\footnote{Research supported by the Swiss National Science Foundation (Project SNSF-PP00P2-138948).\newline An extended abstract of this paper \cite{ExtAbs} appeared in the Proceedings of the 32nd International Symposium on Computational Geometry (SoCG 2016).}}
\author[1]{Dominic Dotterrer}
\author[2]{Tali Kaufman}
\author[3]{Uli Wagner}
\affil[1]{
University of Chicago
Department of Mathematics
5734 S.~University Avenue, Chicago, Illinois 60637, USA.\\ \texttt{d.dotterrer@math.uchicago.edu}
}
\affil[2]{
Bar-Ilan University, Department of Computer Science, 5290002 Ramat Gan, Israel.\\ \texttt{kaufmant@macs.biu.ac.il}
}
\affil[3]{
IST Austria, Am Campus 1, 3400 Klosterneuburg, Austria.\\ \texttt{uli@ist.ac.at}
}

\begin{document}
\maketitle

\begin{abstract}
We give a detailed and easily accessible proof of Gromov's \emph{Topological Overlap Theorem}.
Let $X$ be a finite simplicial complex or, more generally, a finite polyhedral cell complex of dimension $d$.
Informally, the theorem states that if $X$ has sufficiently strong \emph{higher-dimensional expansion properties} (which generalize
edge expansion of graphs and are defined in terms of cellular cochains of $X$) then $X$ has the following \emph{topological overlap property}: 
for every continuous map $X\rightarrow \R^d$ there exists a point $p\in \R^d$ that is contained in the images of a positive fraction $\overlap>0$ of the $d$-cells of $X$. 
More generally, the conclusion holds if $\R^d$ is replaced by any $d$-dimensional piecewise-linear (PL) manifold $M$, with a constant $\overlap$
that depends only on $d$ and on the expansion properties of $X$, but not on $M$.
\end{abstract}

\section{Introduction}
Let $X$ be a finite polyhedral cell complex\footnote{See \cite[Sec.~12]{Bjorner} or \cite[Ch.~I]{Zeeman} for more background on polyhedral cell complexes (in \cite{Zeeman}, they are called convex linear cell complexes).} of dimension $\dim X=d$. Gromov~\cite{Gromov} recently showed that if $X$ has
sufficiently strong \emph{higher-dimensional expansion properties} (which generalize edge expansion of graphs, see below for the precise definition) then $X$ has the following \emph{topological overlap property}:
For every every continuous map $\map \colon X\to \R^d$, there exists a point $p\in \R^d$ that is contained in the images of 
some positive fraction of the $d$-cells of $X$, i.e.,
\begin{equation}
\label{eq:overlap}
\hfill |\{\sigma\in\faces{d}{X}\colon p\in f(\sigma)\}|\geq \overlap\cdot |\faces{d}{X}|,\hfill
\end{equation}
where $\faces{k}{X}$ denotes the set of $k$-dimensional cells of $X$, $0\leq k\leq d$, and $\overlap>0$.
More generally, the same conclusion holds if the target space $\R^d$ is replaced by a $d$-dimensional manifold $M$, and the \emph{overlap constant} $\mu>0$ depends only on the dimension $d$ and on the constants quantifying the expansion properties of $X$, but not on $M$. 
For technical reasons, we will assume that the manifold $M$ admits a \emph{piecewise-linear} (\emph{PL}) triangulation, so that we can apply standard tools to perturb a given map to \emph{general position}. We refer to the book by Rourke and Sanderson~\cite{book} or to the lecture notes by Zeeman~\cite{Zeeman} for background and standard facts about piecewise-linear topology.

In the special case where $X$ is the $n$-dimensional simplex $\Delta^n$ (or its $d$-dimensional skeleton), determining the optimal overlap constant for maps $\Delta^n\to \R^d$ is a classical problem in discrete geometry, also known as the \emph{point selection problem} \cite{BF,Barany} and originally only considered for affine maps. Apart from the generalization from affine to arbitrary continuous
maps, Gromov's proof also led to improved estimates for the point selection problem, and a number of papers have appeared with expositions and simplified proofs of Gromov's result in this special case $X=\Delta^n$, see \cite{Karasev,MW} and \cite[Sec.~7.8]{Guth}.

The goal of the present paper is to provide a detailed and easily accessible proof of Gromov's result for general complexes $X$, see Theorem~\ref{thm:overlap} below. This is a crucial ingredient for obtaining examples of simplicial complexes $X$ of \emph{bounded degree} (i.e., such that every vertex is incident to a bounded number of simplices) that have the topological overlap property~\cite{KKL,EK}. The basic idea of the proof is the same as Gromov's, but we present a simplified and streamlined version of the proof that uses only elementary topological notions (general position for piecewise-linear maps, algebraic intersection numbers, cellular chains and cochains, and chain homotopies) and avoids much of the machinery  used in Gromov's original paper (in particular, the simplicial set of cocycles). 

For stating the result formally, we need to discuss higher-dimensional expansion properties of cell complexes. The relevant notion of expansion originated in the work of Linial and Meshulam~\cite{LM} and of Gromov~\cite{Gromov} and generalizes
edge expansion of graphs (which corresponds to $1$-dimensional expansion). To define $k$-dimensional expansion,
we need two ingredients: first, information about incidences between cells of dimensions $k$ and $k-1$ and, second,
a notion of \emph{discrete volumes} in $X$. To define these, it is convenient to use the language of \emph{cellular cochains}
of $X$.

\subsubsection*{Cellular Cochains}
Let $X$ be a polyhedral cell complex, let $\faces{k}{X}$ denote the set of $k$-dimensional cells of $X$, and let $C^k(X):=C^k(X;\F_2):=\F_2^{\faces{k}{X}}$ be the space of $k$-dimensional cellular cochains with coefficients in the field $\F_2$; in other words $C^k(X)$ is the space of functions $a\colon \faces{k}{X}\to \F_2=\{0,1\}$. For a pair $(\sigma,\tau) \in \faces{k}{X}\times \faces{k-1}{X}$, let $\incident{\sigma}{\tau}$ be $1$ or $0$ depending on whether $\tau$ is incident to $\sigma$ (i.e., whether $\tau$ is contained in the boundary $\partial \sigma$) or not. This incidence information is recorded in the \emph{coboundary operator}, which is a linear map
$\delta \colon C^{k-1}(X)\to C^k(X)$ given by $\delta a(\sigma):=\sum_{\tau\in \faces{k-1}{X}} \incident{\sigma}{\tau}a(\tau)$.

The elements of the subspaces $Z^k(X):=\ker(\delta\colon C^k(X)\to C^{k+1}(X))$ and $B^k(X):=\im (\delta\colon C^{k-1}(X)\to C^k(X))$ are called $k$-dimensional \emph{cocycles} and \emph{coboundaries}, respectively. The composition of consecutive coboundary operators is zero, i.e., $B^k(X)\subseteq Z^k(X)$, and $H^k(X)=Z^k(X)/B^k(X)$ is the $k$-dimensional \emph{homology group} (with $\F_2$-coefficients) of $X$.
This information is customarily recorded in the \emph{cellular cochain complex}\footnote{More precisely, we work with the \emph{augmented} cellular cochain complex of $X$, unless stated otherwise, i.e., we consider $X$ to have a unique $(-1)$-dimensional cell, the empty cell $\emptyset$, which is incident to every vertex of $X$.} of $X$:

\begin{equation}
\label{eq:cochain-complex}
\xymatrix@C=4ex{
0 \ar[r] 
&
\F_2=C^{-1}(X) \ar[r]^-{\delta}
&
C^0(X) \ar[r]^-{\delta}
&
C^1(X) \ar[r]^-{\delta}
&
\cdots \ar[r]^-{\delta}
&
C^{d-1} (X) \ar[r]^-{\delta}
&
C^d (X) \ar[r]
&
0
}
\end{equation}

\subsubsection*{Norm, cofilling, expansion and systoles}

For $\alpha\in C^k(X)$, let $|\alpha|$ denote the \emph{Hamming norm} of $\alpha$, i.e., the cardinality of the \emph{support}
$\supp(\alpha):=\{\sigma \in \faces{k}{X}\colon \alpha (\sigma)\neq 0\}$, which we think of as a measure of ``discrete $k$-dimensional volume.'' In fact, it will be convenient to allow more general norms on cochains; the following definition summarizes the properties that we will need.

\begin{definition}[Norm on cochains]
A \emph{norm} on the group $C^\ast(X)=\bigoplus_{k=0}^d C^k(X)$ of cellular cochains of $X$ with $\F_2$-coefficients is a function ${\|\cdot \|}\colon C^\ast(X;\F_2)\rightarrow \R_{\geq 0}$ that satisfies the following properties for all cochains $\alpha,\beta\in C^k(X)$, $0\leq k\leq d$:
\begin{enumerate}
\item $\|0\|=0$.\smallskip
\item \emph{Triangle inequality}:\quad $\|\alpha+\beta\|\leq \|\alpha\| +\|\beta\|$.
\end{enumerate}
\smallskip
Furthermore, we will assume throughout that the norm satisfies the following:
\begin{enumerate}
\setcounter{enumi}{2}
\item \emph{Monotonicty}:\quad $\|\alpha\|\leq \|\beta\|$ whenever $\supp(\alpha)\subseteq \supp(\beta)$.
\end{enumerate}
\end{definition}

From now on, we work with a fixed norm on the cochains of $X$. We assume that the norm is normalized such that $\|\I_X^k\|=1$ for $0\leq k\leq d$, where $\I_X^k\in C^k(X)$ assigns $1$ to every $k$-cell of $X$.
In particular, when working with the Hamming norm, we will consider its normalized version
$$\|\alpha\|_H:=\frac{|\alpha|}{|\faces{k}{X}|}.$$

Given $\beta \in B^k(X)$, we say that $\alpha\in C^{k-1}(X)$ \emph{cofills} $b$ if $\beta=\delta \alpha$.  
Once we have a notion of discrete volumes, we can consider the following \emph{(co)isoperimetric question}:
Can we bound the minimum norm of a cofilling for a coboundary $\beta$ in terms of the norm of $\beta$?

\begin{definition}[Cofilling/Coisoperimetric Inequality] Let $\isobound>0$.
We say that $X$ satisfies a $L$-\emph{cofilling inequality} (or \emph{coisoperimetric inequality}) in dimension $k$ if, for every $\beta\in B^k(X)$, there exists some $\alpha\in C^{k-1}(X)$ such that $\delta\alpha=\beta$ and $\|\alpha\|\leq \isobound \|\beta\|$.
\end{definition}

Any two cofillings of a given coboundary differ by a cocycle. Thus, $X$ satisfies an $L$-cofilling inequality in dimension $k$ if and only if
\begin{equation}
\label{eq:-coiso-reformulation-1}
\hfill \|\delta\alpha\|\geq \frac{1}{\isobound} \cdot \min\{\|\alpha+\zeta\|:\zeta\in Z^{k-1}(X)\}\qquad \textrm{for all}\quad \alpha\in C^{k-1}(X). \hfill
\end{equation}

We can strengthen \eqref{eq:-coiso-reformulation-1} by replacing cocycles with coboundaries and obtain a condition that also allows us to draw conclusions about the cohomology of $X$. For $\alpha \in C^{k-1}(X)$, let
\begin{equation}
\label{eq:dist-cobd}
\hfill \|[\alpha]\|:=\min\{\|\alpha+\beta\|:\beta \in B^{k-1}(X)\}\hfill
\end{equation}
denote the distance (with respect to the norm $\|\cdot\|$) of $\alpha$ to the space $B^{k-1}(X)$ of coboundaries.

\begin{definition}[Coboundary Expansion] Let $\expfactor>0$.
We say that $X$ is \emph{$\expfactor$-expanding} in dimension $k$, if for every $(k-1)$-cochain $\alpha \in C^{k-1}(X)$,
\begin{equation}
\label{eq:expansion}
\hfill \|\delta \alpha\|\geq \expfactor\cdot \|[\alpha]\|. \hfill
\end{equation}
\end{definition}

\begin{lemma} Let $\eta>0$.
A complex $X$ is $\expfactor$-expanding in dimension $k$ if and only if $H^{k-1}(X)=0$ and $X$ satisfies a $1/\expfactor$-coisoperimetric inequality in dimension $k$.
\end{lemma}
\begin{proof}
Suppose that $X$ is $\expfactor$-expanding in dimension $k$.
Clearly, \eqref{eq:expansion} implies \eqref{eq:-coiso-reformulation-1}, i.e., $X$ satisfies a $1/\expfactor$-cofilling inequality.
Moreover, if $\alpha \in C^{k-1}(X)\setminus B^{k-1}(X)$ then $\|[\alpha]\|>0$, hence $\|\delta \alpha\|>0$, hence $\alpha\not\in Z^{k-1}(X)$. Thus, $Z^{k-1}(X)=B^{k-1}(X)$, i.e., $H^{k-1}(X)=0$.

Conversely, assume that $H^{k-1}(X)=0$. Then $Z^{k-1}(X)=B^{k-1}(X)$, so \eqref{eq:expansion} and \eqref{eq:-coiso-reformulation-1} are equivalent.
\end{proof}

In some cases, however, vanishing of $H^{k-1}(X)$ turns out to be too stringent a requirement, and we can replace it by the condition that every nontrivial cocycle has large norm:

\begin{definition}[Large Cosystoles] Let $\sysbound>0$.
We say that $X$ has \emph{$\sysbound$-large cosystoles} in dimension $j$ if $\|\alpha\|\geq \sysbound$ for
every $\alpha\in Z^j(X)\setminus B^j(X)$.
\end{definition}

\begin{example}
Consider the case $k=1$, with the normalized Hamming norm. In this case, $\expfactor$-expansion in dimension $1$ corresponds to $\expfactor$-edge expansion of a graph (the $1$-skeleton of the complex). An $\isobound$-cofilling inequality in dimension $1$ means that every connected component of the graph is $1/\isobound$-edge expanding. Having $\sysbound$-large cosystoles in dimension $0$ means that every connected component contains at least a $\sysbound$-fraction of the vertices.
\end{example}

\subsubsection*{Local Sparsity of $X$}
For the formal statement of the overlap theorem, we need one more technical condition on $X$.
For a cell $\tau$ of $X$, let $\iota^k_\tau$ be the $k$-dimensional cochain that assigns $1$ to $k$-cells of $X$ that have nonempty intersection with $\tau$ and $0$ otherwise.

\begin{definition}(Local Sparsity) Let $\varepsilon>0$. We say that $X$ is \emph{locally $\sparsity$-sparse} (with respect to a given norm $\|\cdot\|$) if $\|\iota^k_\tau\| \leq \sparsity$ for every nonempty cell $\tau$ of $X$ and every $k$, $0\leq k\leq d$.
\end{definition}

For example, in the case of the normalized Hamming norm $\|\cdot\|_H$, local sparsity means that
$$|\{\sigma \in \faces{k}{X}\colon \tau\cap \sigma\neq \emptyset \}| \leq \sparsity |\faces{k}{X}|,$$
for every nonempty cell $\tau$ of $X$.

\subsubsection*{Formal Statement of the Theorem}
We are now ready to state Gromov's theorem.

\begin{theorem}[Gromov's Topological Overlap Theorem~\cite{Gromov})]
\label{thm:overlap}
For every $d\geq 1$ and $\isobound,\sysbound>0$ there exists $\sparsity_0=\sparsity_0(d,\isobound,\sysbound)>0$ such that the following holds:

Let $X$ be a finite cell complex of dimension $d$, and let $\|\cdot\|$ be a norm on the cochains of $X$.
Suppose that
\begin{enumerate}
\item $X$ satisfies a $\isobound$-cofilling inequality in dimensions $1,\dots,d$;
\item $X$ has $\sysbound$-large cosystoles in dimensions $0,\dots,d-1$; and
\item $X$ is locally $\sparsity$-sparse for some $\sparsity \leq \sparsity_0$.
\end{enumerate}

Then for every continuous map $\map\colon X\rightarrow M$ into a compact connected
$d$-dimensional piecewise-linear (PL) manifold $M$,
there exists a point $p\in M$ such that\footnote{Here, we use that a subset of $\Sigma_k(X)$ can be identified with a
$k$-dimensional cellular cochain, its indicator function.}
\begin{equation}
\label{eq:norm-overlap}
\hfill \|\{\sigma \in \Sigma_d(X) \mid p \in f(\sigma) \} \| \geq \overlap, \hfill
\end{equation}
where $\overlap=\overlap(d, \sparsity, \isobound, \sysbound)>0$.
\end{theorem}

\begin{remark}
The assumption that the manifold $M$ is compact is not essential; moreover, we may assume without loss of generality that $M$ has no boundary. Indeed, since $X$ is compact, the image $\map(X)$ is compact and hence
contained in a compact submanifold $N$ of $M$ with boundary $\partial N$; we can turn $N$ into a compact manifold without boundary by doubling, i.e., by glueing two copies of $N$ along their boundary.
\end{remark}

If a complex $X$ satisfies the conclusion of the theorem, we also say that $X$ is \emph{topologically $\overlap$-overlapping} for maps into $d$-dimensional PL manifolds. If the conclusion holds true just for \emph{affine maps} and $M=\R^d$, we say that
$X$ is \emph{geometrically $\overlap$-overlapping}.

\section{Preliminaries from Piecewise-Linear Topology}

\subsection{Assumptions on $M$}

We assume that $M$ is a compact connected piecewise-linear (PL) $d$-dimensional manifold, without boundary. That is, we assume that $M$ admits a triangulation\footnote{The triangulation is necessarily finite, since $M$ is compact.} $T$ with the property that the link of every nonempty simplex $\tau$ of $T$ is a PL sphere of dimension $d-1-\dim(\tau)$; throughout this paper, we only consider triangulations of $M$ that have this property.


\subsection{Approximation by PL maps}
\label{sec:approx}

We can fix a metric on $M$, e.g., by fixing a triangulation $T$ of $M$ and by considering each simplex of $T$ as a regular simplex with edge length $1$. By subdividing a given triangulation $T$ sufficiently often, we can pass to a
new triangulation $T'$ in which each simplex has diameter at most $\rho>0$, for a given $\rho$ (see, e.g., \cite[Sec.~1.7]{Mat}).

By the standard \emph{simplicial approximation theorem}~\cite{Pras}, given the triangulation $T'$ of $M$ and a continuous map $\map \colon X\to M$, there is a simplicial approximation of $\map$, i.e., there is a subdivision $X'$ of $X$ and a simplicial map $g\colon X'\to T'$ such that, for each point $x \in X$, the image $g(x)$ belongs to the (uniquely defined) simplex of $T'$ whose relative interior contains $\map(x)$. (In fact, $g$ is even homotopic to $\map$, but we will not need that.)
This map $g$ is a PL map $X\to M$ and the distance between $g(x)$ and $\map(x)$ is at most the maximum diameter of any simplex in $T'$, hence at most $\rho$, for every $x\in X$.

Thus, by the preceding discussion and the following lemma, it suffices to prove Thm.~\ref{thm:overlap} for PL maps.
\begin{lemma}
Let $\map \colon X\to M$ be a continuous map, and let $g_n\colon X\to M$ be a sequence of continuous maps that converges to $f$ pointwise, i.e., $g_n(x) \to f(x)$ as $n\to \infty$, for every $x\in X$. Suppose that for every $g_n$ there exists a point $p_n\in M$ such that $\|\{\sigma \in \Sigma_d(X) \mid p_n \in g_n(\sigma) \} \| \geq \overlap$. Then there exists a point $p\in M$ such that \eqref{eq:norm-overlap} holds.
\end{lemma}
\begin{proof}
By compactness, there is a subsequence of the points $p_n$ that converges to a point $p$. We claim that $p$ is the desired point. Since there are only finitely many cells in $X$, there is some $\rho>0$ such that for every $d$-cell $\sigma$ of $X$ with
$p\not \in \map(\sigma)$, the distance between $p$ and $f(\sigma)$ is at lest $\rho$. Choose $n$ sufficiently large so that
the distance between $p_n$ and $p$ is less than $\rho/2$, and the distance between $f(x)$ and $g_n(x)$ is at most $\rho/2$, for every $x\in X$. If $p_n\in g_n(\sigma)$, then the distance between $p$ and $f(\sigma)$ is less than $\rho$, so by the choice of $\rho$, we have $p\in f(\sigma)$. Therefore, $\{\sigma \in \Sigma_d(X) \mid p \in f(\sigma) \} \subseteq \{\sigma \in \Sigma_d(X) \mid p_n \in g_n(\sigma) \}$, and the desired conclusion follows by the monotonicity property of the norm.
\end{proof}

\subsection{General Position}
We refer to \cite[Ch. VI]{Zeeman} for a comprehensive treatment of general position for PL maps. The following definition summarizes the properties that we will need. 
\begin{definition}
Let $X$ be a finite polyhedral cell complex, $M$ a PL manifold, and let $f\colon X\to M$ be a PL map. 
\begin{enumerate}
\item
We say that $f$ is in \emph{strongly general position} (with respect to the given decomposition of $X$ into polyhedral cells) if, for every $r\geq 1$ and  
pairwise disjoint cells $\sigma_1,\ldots,\sigma_r$ of $X$,
\begin{equation}
\label{eq:strongly-general-position-f}
\hfill \dim\big({\textstyle \bigcap_{i=1}^r \map(\sigma_i)} \big) \leq \max \big\{-1, \big({\textstyle \sum_{i=1}^r} \dim \sigma_i \big) - d(r-1)\big\}.
\hfill
\end{equation}
In particular, if the number of the right-hand side is $-1$, then the intersection is empty.
\item Given a triangulation $T$ of $M$, we that that $f$ is in \emph{general position with respect to $T$} if, for every simplex $\sigma$ of $X$ and every simplex $\tau$ of $T$, $\dim(\map(\sigma)\cap \tau)\leq \max\{-1,\dim \sigma +\dim \tau -d\}$; moreover, if $\dim \sigma+\dim \tau=d$ then we require that $f(\sigma)$ and $\tau$ intersect transversely (either the intersection is empty, or they intersect locally like complementary linear subspaces).
\end{enumerate}
\end{definition}

The main fact that we will need is that any map $f\colon X\to M$ can be approximated arbitrarily closely by a PL map that is in general position:

\begin{lemma}[{\cite[Ch. VI]{Zeeman}}]
\label{lem:GP}
Let $f\colon X\to M$ be a PL map and let $T$ be a triangulation of $M$. Then, up to a small perturbation, we may assume that 
$f$ is  general position with respect to $T$ and in strongly general position.
\end{lemma}

Furthermore, we will need the following notion of \emph{sufficiently fine triangulations}:
\begin{definition}
Let $T$ be a triangulation of $M$ and let $\map \colon X\to M$ be a PL map in general position with respect to $T$. 
We say that $T$ is \emph{sufficiently fine} with respect to $f$ if, for every $k>0$ and every $k$-simplex $\tau$ of $T$,
$$\|\{\sigma\in \faces{d-k}{X}\colon f(\sigma)\cap \tau\neq \emptyset\}\| \leq \frac{d}{k} \max\{ \|\iota^{d-k}_{\sigma'}\| \colon \sigma' \in \faces{d-k}{X}\}.$$
\end{definition}
\begin{lemma}
\label{lem:suff-fine}
Suppose that $\map \colon X\to M$ be a PL map in strongly general position and in general position with respect to a triangulation $T$ of $M$. Then (by refining $T$, if necessary), we may assume furthermore that $T$ is sufficiently fine
with respect to $f$.
\end{lemma}
\begin{proof}
If $\map$ is in general position with respect to $T$, then by choosing points at which we subdivide $T$ in a sufficiently generic way, we can assume that $f$ is also in general position with respect to the subdivision $T'$. Thus, we may assume that $T$ already has the property that every simplex of $T$ has diameter smaller than some specified parameter $\rho>0$.

Now suppose that $\sigma_1,\ldots,\sigma_r$ are pairwise distinct simplices of $X$ with $f(\sigma_1)\cap \ldots \cap f(\sigma_r)=\emptyset$. By compactness, there exists $\rho=\rho(\sigma_1,\ldots,\sigma_r)>0$ such that no matter how we select $x_i\in f(\sigma_i)$, some pair $x_i,x_j$ has distance at least $\rho$. Since $X$ is finite, there is some $\rho>0$ that works for all finite collections of simplices whose images do not have a common point of intersection. Suppose now that we have chosen $T$ such that all simplices in $T$ have diameter at most $\rho/2$. 

Given $\tau\in T$ of dimension $k>0$ consider $S(\tau):=\{\sigma\in \faces{d-k}{X}\colon f(\sigma)\cap \tau\neq \emptyset\}$.
We claim that $\bigcap_{\sigma\in S(\tau)} f(\sigma)\neq \emptyset$. Otherwise, for every choice of points $x_\sigma \in f(\sigma)$, $\sigma\in S(\tau)$, there would be some pair $\sigma,\sigma'$ such that $x_\sigma$ and $x_{\sigma'}$ have distance at least $\rho$. However, by the definition of $S(\tau)$, we can choose each $x_\sigma$ to lie in the intersection $f(\sigma)\cap \tau$, from which it follows that for every pair $\sigma,\sigma' \in S(\tau)$, the distance between $x_\sigma$ and $x_{\sigma'}$ is at most the diameter of $\tau$, i.e., at most $\rho/2$.

Let $\{\sigma_1,\ldots,\sigma_r\} \subseteq S(\tau)$ be an inclusion-maximal subset with $\sigma_i\cap \sigma_j=\emptyset$ (i.e., the $\sigma_i$ are pairwise vertex-disjoint; we can pick this subset greedily). Since $f$ is in strongly general position and  
$\bigcap_{\sigma\in S(\tau)} f(\sigma)\neq \emptyset$, it follows that $\sum_{i=1}^r(d-k) - d(r-1) \geq 0$; this implies $r\leq d/k$. Now, every other simplex $\sigma \in S(\tau)$ 
intersects one of the $\sigma_i$. Thus, by monotonicity of the norm and by the triangle inequality, $\|S(\tau)\|\leq \frac{d}{k} \max_{1\leq i\leq r} \|\iota^{d-k}_{\sigma_i}\|$.
\end{proof}

\subsection{Intersection Numbers}

\begin{definition}[Intersection numbers]
If $T$ is a PL triangulation of $M$ and if $\map\colon X\rightarrow M$ is a PL map in general position with respect to $T$, then for every pair of chains $a\in C_{d-k}(X;\F_2)$ and $b\in C_k(T;\F_2)$, we can define their (algebraic) \emph{intersection number}
$$\map(a)\cdot b \in \F_2$$
as follows: If $\sigma$ is a $(d-k)$-dimensional cell of $X$ and if $\tau$ is a $k$-dimensional simplex of $T$, 
then by general position, the intersection $f(\sigma)\cap \tau$ consists of a finite number of points, and the
intersection number $f(\sigma)\cdot \tau$ is defined as the number of intersections\footnote{There is a small
caveat: In the case $k=0$, an intersection point in $f(\sigma)\cdot \tau$ may have several preimages in $\sigma$
and should be counted with the corresponding multiplicity; equivalently, the intersection number is defined as the 
number of points in $\sigma\cap f^{-1}(\tau)$ modulo $2$.} modulo $2$. This definition is extended by linearity (over $\F_2$)
to arbitrary chains.

This yields, for $0\leq k\leq d$, an \emph{intersection number homomorphism} 
\begin{equation}
\hfill \map^\pitchfork\colon C_k(T) \rightarrow C^{d-k}(X),\hfill
\end{equation}
defined by $\map^\pitchfork(b)(a)=\map(a)\cdot b$ for each $a\in C_{d-k}(X).$
\end{definition}

It is well-known that the intersection number homomorphism is a \emph{chain-cochain map}, i.e., it commutes with the boundary and coboundary operators in the following sense (see, e.g., \cite[Sec.~2.2]{MabW} for a detailed review of this and other properties of intersection numbers).
\begin{lemma}
$$\map^\pitchfork(\partial a) =\delta\map^\pitchfork(a).$$
\end{lemma}

For the proof of the main theorem, we need the following definition:

\begin{definition}[Chain-cochain homotopy]  Consider two chain-cochain maps
$\varphi,\psi \colon C_k(M)\to C^{d-k}(X)$ from the (non-augmented) chain complex of $M$ to the cochain complex of $X$.
A \emph{chain-cochain homotopy} between $\varphi$ and $\psi$ is a family of linear maps $h\colon C_k(M)\to C^{d-k-1}(X)$
such that  $\varphi-\psi=h\partial + \delta h$. To keep track of the various maps, it is convenient to keep in mind the following diagram:
\begin{equation}
\label{eq:chain-homotopy}
\xymatrix{
0 \ar[r]
&
C_d(M) \ar@<-.5ex>[d]_{\varphi} \ar@<.5ex>[d]^{\psi}
\ar[r]^{\partial}
\ar[dl]^{h}
&
C_{d-1}(M) \ar@<-.5ex>[d]_{\varphi} \ar@<.5ex>[d]^{\psi}
\ar[r]^{\partial}
\ar[dl]^{h}
&
\ar[dl]^{h}
\cdots
\ar[r]^{\partial}
&
C_{1}(M) \ar@<-.5ex>[d]_{\varphi} \ar@<.5ex>[d]^{\psi}
\ar[r]^{\partial}
\ar[dl]^{h}
&
C_{0}(M) \ar@<-.5ex>[d]_{\varphi} \ar@<.5ex>[d]^{\psi}
\ar[r]
\ar[dl]^{h}
&
0
\\
0 \ar[r]
&
C^0(X) \ar[r]_{\delta}
&
C^1(X) \ar[r]_{\delta}
&
\cdots \ar[r]_{\delta}
&
C^{d-1} (X) \ar[r]_{\delta}
&
C^d (X) \ar[r]
&
0
}
\end{equation}
\end{definition}

\section{Proof of the Overlap Theorem}

\begin{proof}[Proof of Theorem~\ref{thm:overlap}]
Let $\mu$ and $\sparsity_0$ be parameters that we will determine in the course of the proof. We assume that $X$ satisfies the assumptions of the theorem, in particular that it is locally $\sparsity$-sparse for some $\sparsity\leq \sparsity_0$.

Let $\map \colon X\to M$ be a map. By the discussion in Sec.~\ref{sec:approx} and by Lemmas~\ref{lem:GP} and \ref{lem:suff-fine}, we may assume that $f$ is PL and in general position with respect to a sufficiently fine PL triangulation $T$ of $M$.
 
We wish to show that there is a vertex $v$ of $T$ such that the intersection number cochain $\map^\pitchfork(v) \in C^d(X)$ satisfies $\|\map^\pitchfork(v)\| \geq \overlap$. We assume that this is not the case and we proceed to derive a contradiction.

Let $v_0$ be a fixed vertex of $T$; by assumption, $\|\map^\pitchfork(v_0)\|<\overlap$. (Note that if $f$ is not surjective then we can choose the triangulation $T$ and $v_0$ so that $\|\map^\pitchfork(v_0)\|=0$.)

We define a chain-cochain map\footnote{That is, a homomorphism $G\colon C_k(T)\rightarrow C^{d-k}(X)$ for every $k$ such that $G(\partial c)=\delta G(c)$ for $c\in C_k(T)$.}
$$G\colon C_\ast(T)\rightarrow C^{d-\ast}(X)$$
by setting $G(v):=\map^\pitchfork(v_0)$ for every vertex $v$ of $T$ and $G(c)=0$ for every $c\in C_k(T;{\F_2})$, $k>0$.

We will construct a \emph{chain-cochain homotopy} $H\colon C_\ast(T)\rightarrow C^{d-1-\ast}(X)$ between $\map^\pitchfork$ and $G$; that is, for every $k$, we construct a homomorphism
$$H\colon  C_k(T)\rightarrow C^{d-1-k}(X)$$
such that
\begin{equation}
\label{eq:homotopy-condition}
\hfill \map^\pitchfork(c) - G(c)= H(\partial c) + \delta H(c) \hfill
\end{equation}
for $c\in C_k(T)$. We stress that for this proof, we work with \emph{non-augmented} chain and cochain complexes as in \eqref{eq:chain-homotopy}, i.e., we use the convention that $C^{-1}(X)=0$. 
It follows that $G(c)=0$ for $k>0$ and that $H(c)=0$ for $c\in C_d(M)$.

The chain-cochain homotopy $H$ will yield the desired contradiction: Given the triangulation $T$ of $M$, the formal sum of all $d$-dimensional simplices of $T$ is a $d$-dimensional cycle $\zeta_M$ (here we use that $M$ has no boundary). Note that $\map^\pitchfork(\zeta_M)=\I^0_X$
(every vertex $v$ of $X$ is mapped into the interior of a unique $d$-simplex of $M$) but $G(\zeta_M)=0$. This is a contradiction, since
$$0\neq \I^0_X = \map^\pitchfork(\zeta_M) - G(\zeta_M) = \underbrace{H(\partial \zeta_M)}_{=0\textrm{ since }\partial \zeta_M=0} + \delta \underbrace{H(\zeta_M)}_{=0}=0.$$

To complete the proof, it remains construct $H$, which we will do by induction on $k$.

For $k=0$, we observe that for every vertex $v$ of $T$, the cochains $\map^\pitchfork(v)$ and $G(v)=\map^\pitchfork(v_0)$
are cohomologous, i.e., their difference is a coboundary: We assume that $M$ is connected, hence there is a $1$-chain (indeed, a path) $c$ in $T$ with $\partial c=v-v_0$, and so $\map^\pitchfork(v)-G(v)=\map^\pitchfork(v-v_0)=\delta \map^\pitchfork(c)$.
For every vertex $v$ of $T$, we set $H(v)$ to be a cofilling of $\map^\pitchfork(v)-G(v)$ of minimal norm (if there is more than one minimal cofilling, we choose one arbitrarily). Thus, the homotopy condition \eqref{eq:homotopy-condition} is satisfied for $0$-chains (since chains and cochains of dimension less than zero or larger than $d$ are, by convention, zero).

By choice of $H(v)$ and the coisoperimetric assumption on $X$, we have
$$\|H(v)\|\leq \isobound \underbrace{\|\map^\pitchfork(v)-\map^\pitchfork(v_0)\|}_{<2\overlap}<s_0:=2\isobound \overlap.$$

Inductively, assume that we have already defined $H$ on chains of dimension less than $k$ and that $\|H(\rho)\|<s_i$ for every $i$-simplex of $T$, $i<k$, where $s_i$ is a parameter that we will determine inductively. Thus, if $\tau$ is a $k$-simplex of $T$, then $H(\partial \tau)$ is already defined and has norm less than $(k+1)s_{k-1}$.

Moreover, we have $\|\map^\pitchfork(\tau)\|\leq \frac{d}{k}\sparsity \leq
d \sparsity$, by the sparsity assumption on $X$ and since the triangulation $T$ is sufficiently fine.

By construction, $z:=\map^\pitchfork(\tau)- H(\partial \tau)$ is a $(d-k)$-dimensional cocycle, and
\begin{equation}\label{eq:bd-z}
\hfill \|z\|\leq \|\map^\pitchfork(\tau)- H(\partial \tau)\| <d \sparsity + (k+1)s_{k-1}. \hfill
\end{equation}

If $z$ is cohomologically trivial, i.e., $z\in B^{d-k}(X)$, then we define $H(\tau)$ to be a minimal cofilling of $z$ and extend $H$ to $C_k(T)$ by linearity. By assumption on $X$, we get 
$$\|H(\tau)\| < s_{k}:=\isobound\left( 
d \sparsity + (k+1)s_{k-1}\right).$$
Note that this recursion yields $s_k =
d \sparsity(\isobound+\dots+\isobound^k)+(k+1)! \isobound^{k+1} 2\overlap.$

If $z$ is nontrivial,\footnote{Note that in the special case that $X$ is connected and $k=d$, the only nontrivial $0$-cocycle is $z=\I_X^0$, hence $\|z\|=1$.} then by the assumption on large cosystoles and \eqref{eq:bd-z},
$$\sysbound \leq \|z\| < d \sparsity + (k+1)s_{k-1},$$
which is a contradiction if we choose $\overlap$ and $\sparsity_0$ (and hence $\sparsity$) sufficiently small with respect to $d$, $\isobound$ and $\sysbound$.
\end{proof}

\begin{remarks}
\begin{enumerate}
\item In many interesting cases, $X$ belongs to an infinite family of complexes for which the local sparsity parameter $\sparsity$
tends to zero as the size of the complex increases. For instance, if $X$ is the $d$-skeleton of the $n$-simplex, $n\to \infty$, then
we have $\sparsity=O(1/n)$. For complexes with local sparsity $\sparsity=o(1)$, the above proof yields
$\mu \geq \frac{\sysbound}{2(k+1)!L^k}+o(1)$.
If $M$ is unbounded, then, as remarked in the proof, we can take the vertex $v_0$ to satisfy $\map^\pitchfork(v_0)=0$, which improves the estimate by a factor of $2$.

More quantitative information and better bounds on the overlap constant (which are of interest for specific families of complexes, e.g., skeleta of simplices) can be gleaned from the proof by a more refined analysis through the \emph{cofilling profiles} of $X$ \cite{Gromov}, which estimate the size of a minimal cofilling of a cocycle $b$ as a possibly nonlinear function of $\|b\|$. Further improvements in the estimates are possible trough the notion of \emph{pagodas} \cite{MW}.

\item The proof of the overlap theorem is very robust and easily generalizes to other settings, in particular to other coefficient rings and other norms. Suppose that $R$ is a fixed ring of coefficients (commutative, with $1$), and consider (co)chains and (co)homology with $R$-coefficients. If $R$ is not of characteristic $2$, we need to add some minor assumptions to deal with orientations. First, we need to assume that he target manifold $M$ is $R$-orientable, i.e., that $H_d(M;R)\cong R$, generated by a fundamental homology class $[M]$.
The definition of the intersection number changes slightly: if two oriented linear simplices $\sigma, \tau$ of complementary dimensions in $M$ intersect transversely in a single point, then their orientations determine a local orientation of $M$, and we set the intersection number $\sigma\cdot \tau$ to be $+1$ or $-1$ depending on whether this orientation agrees with the chosen global orientation of $M$ or not.

Second, we need to assume that the norm of a cochain is invariant under sign changes in the values of the cochain, i.e., if two $k$-cochains $c,c'\in C^k(X;R)$ satisfy $c(\sigma)=\pm c'(\sigma)$ for every orientated $k$-cell $\sigma$ of $X$ (the signs may be different for different $\sigma$), then $\|c\|=\|c\|$.

With these additional assumptions, the proof of Theorem~\ref{thm:overlap} goes through also for  $R$-coefficients
and yields that for every $\map \colon X\to M$, there exists $p\in M$ such \eqref{eq:norm-overlap} holds.

\item For norms other than the normalized Hamming norm, $\|\map^\pitchfork(p)\|\geq \overlap$ does not necessarily imply that
\eqref{eq:overlap} holds. For instance, suppose that $R=\R$ and that we work with the $\ell_2$-norm. In this case, large norm
$\|\map^\pitchfork(p)\|$ might be caused by a single $d$-simplex $\sigma$ such that $\map^\pitchfork(p)(\sigma)$ is a large integer,
i.e., $\map(\sigma)$ intersects $p$ with large multiplicity. However, this problem does not occur if we impose additional assumptions on
the map $\map$, e.g., that $\map^\pitchfork(p)(\sigma)$ is bounded by some constant $K$ in absolute value (e.g., if $\map$ is linear, then we can take $K=1$).

\item We used the assumption that $M$ is piecewise-linear in order to apply standard general position arguments from piecewise-linear topology. We believe that the result holds more generally if $M$ is a homology manifold.  General position arguments for homology manifolds are much more subtle, but for the proof  we do not really need to perturb the map $\map$ to general position (which may not be possible), we only need a general position chain map that is close to the chain map induced by $f$. We plan to investigate this in more detail in a future paper.
\end{enumerate}
\end{remarks}

\subsubsection*{Acknowledgement.} We would like to thank the anonymous referees 
for many helpful remarks concerning the presentation.

\end{document}